\documentclass[11pt, letterpaper]{amsart}   	
\usepackage{graphicx}				
\usepackage{amssymb}

\usepackage{latexsym,exscale,enumerate,amsfonts,amssymb,mathtools}
\usepackage{amsmath,amsthm,amsfonts,amssymb,amscd,stmaryrd,textcomp}
\usepackage[normalem]{ulem}
\usepackage{young}
\usepackage{thmtools,xcolor}
\usepackage{easybmat}

\definecolor{colormy}{rgb}{0.8,0.05,0.05}
\definecolor{mycolor}{rgb}{0.25,0.99,0.25}

\usepackage[all]{xy}
\SelectTips{cm}{}

\usepackage{tikz}
\usetikzlibrary{decorations.markings}
\usetikzlibrary{decorations.pathreplacing}
\usetikzlibrary{arrows,shapes,positioning}
\tikzstyle directed=[postaction={decorate,decoration={markings,
    mark=at position #1 with {\arrow{>}}}}]
\tikzstyle rdirected=[postaction={decorate,decoration={markings,
    mark=at position #1 with {\arrow{<}}}}]

\usepackage{hyperref}

\newcommand{\Hom}{\mathrm{Hom}}

\newcommand{\Ext}{\mathrm{Ext}}

\newcommand{\mi}{\mathrm{min}}

\newcommand{\spa}{\mathrm{span }}
\newcommand{\supp}{\mathrm{supp }}
\newcommand{\rad}{\mathrm{rad }}

\newcommand{\im}{\mathrm{Im}}

\def\C{{\mathbb C}}

\def\Z{{\mathbb Z}}

\def\cha{\mathrm{ch}}

\theoremstyle{definition}
\newtheorem{thm}{Theorem}[section]
\newtheorem{cor}[thm]{Corollary}
\newtheorem{lem}[thm]{Lemma}
\newtheorem{prop}[thm]{Proposition}

\theoremstyle{definition}
\newtheorem{examplecounter}{Example}

\newtheorem{exa}[thm]{Example}

\numberwithin{equation}{section}

\declaretheorem[style=definition,name=Definition,qed=$\blacktriangle$,numberlike=thm]{defn}
\declaretheorem[style=definition,name=Remark,qed=$\blacktriangle$,numberlike=thm]{rem}

\title{Character formulas in category $\mathcal O_p$}

\author{Henning Haahr Andersen}
\address{Centre for Quantum Mathematics (QM), Imada,
SDU, Denmark}
\email{h.haahr.andersen@gmail.com}

\date{}							

\begin{document}

\begin{abstract}
Let $\mathcal O_p$ denote the characteristic $p>0$ version of the ordinary category $\mathcal O$ for a semisimple complex Lie algebra. In this paper we give some (formal) character formulas in $\mathcal O_p$. First we concentrate on the irreducible characters. Here we give explicit formulas for how to obtain all irreducible characters from the characters of the finitely many restricted simple modules as well as the characters of a small number of infinite dimensional simple modules in $\mathcal O_p$ with specified highest weights. We next prove a strong linkage principle for Verma modules which allow us to split $\mathcal O_p$ into a finite direct sum of linkage classes. There are corresponding translation functors and we use these to further cut down the set of irreducible characters needed for determining all others. Then we show that the twisting functors on $\mathcal O$ carry over to twisting functors on $\mathcal O_p$, and as an application we prove a character sum formula for Jantzen-type filtrations of Verma modules with antidominant highest weights. Finally, we record formulas relating the characters of the two kinds of tilting modules in $\mathcal O_p$.

\end{abstract}

\maketitle

\section{Introduction}
Let $\mathfrak g$ be a semisimple complex Lie algebra. Consider a field $K$ of characteristic $p >0$, 
and let $G$ be the semisimple algebraic group over $K$ corresponding to $\mathfrak g$. Denote by $U_K$ the hyperalgebra over $K$ for $G$, see \cite{RAG}, I.7.7. In \cite{An22} we proved some basic properties of the category $\mathcal O_p$ for $U_K$ as well as for the corresponding quantum version $\mathcal O_q$, when $q$ is a root of unity in $K$. Here we have chosen to focus on $\mathcal O_p$ although many (maybe all) of our results work for $\mathcal O_q$ as well.  We shall use the notation and setup from \cite{An22}, but simplify notation by dropping the bars, i.e. writing just $\mathcal O_p, L_p(\lambda)$, etc. instead of $\bar{\mathcal O_p}, \bar{L}_p(\lambda)$, etc. Along the way we recall some of the key results from \cite{An22} that we need, referring to that paper for further statements and details. Additional background and results can be found in \cite{Di} and \cite{PF}.

We shall use in our context the general concept of (formal) characters introduced in \cite{DGK}, Section 3. This involves in particular working with a completion of the character ring for finite dimensional modules in $\mathcal O_p$, cf. also \cite{PF}, Section 2.5.   Our aim is to deduce some explicit formulas which allow us to determine characters for general modules in $\mathcal O_p$ from the knowledge of the characters of a finite number of simple modules. The necessary background and methods are dealt with in Section 2 where we have also deduced the mentioned formulas for irreducible characters.

Then in Section 3 we explore further the linkage principle for Verma modules from \cite{An22}. We prove a stronger version of this principle using the concept from \cite{DGK} of components (or composition factors) of modules in $\mathcal O_p$. This allow us to decompose $\mathcal O_p$ into a direct sum of linkage classes, one for each of the (finitely many) weights in the closure of the fundamental alcove $\bar C$ for the affine Weyl group. We can then introduce corresponding translation functors on $\mathcal O_p$. As a consequence we then demonstrate that when combined with our formulas in Section 2 we can obtain all irreducible characters from the knowledge of one such character for each facet in $\bar C$.

Another tool available in the study of $\mathcal O$ is the socalled twisting functors, see \cite{Ark}, \cite{AL}, \cite{AS}. In Section 4 we show that we have analogous twisting functors on $\mathcal O_p$ and in Section 5 we use these to deduce a character sum formula for Jantzen-type filtrations of Verma modules with antidominant highest weights. Finally, in Section 6 we give some formulas for tilting characters. In particular, we record the fact, that at least for $p \geq 2h-2$ the characters of $\infty$-tilting modules are determined by the characters of finite dimensional tilting modules.

\section{Characters of simple modules in $\mathcal O_p$} 

Let $n$ denote the rank of $\mathfrak g$ and set $X = \Z^n$. We write $R \subset X$ for the root  system for $\mathfrak g$ and choose a set of positive roots $R^+$. Then we denote by $S = \{\alpha_1, \alpha_2, \cdots , \alpha_n\}$ the set of simple roots in $R^+$. We let $W$ denote the Weyl group for $R$. Then $W$ acts naturally on $X$. We shall also consider the ``dot''-action of $W$ on $X$ given by $w \cdot \lambda = w(\lambda + \rho) - \rho$, $\lambda \in X$. Here $\rho$ is half the sum of all positive roots.

The order on $X$ induced by $S$ is denoted $\leq$ and if $\mu \in X$ we set
$$ X(\leq \mu)= \{\lambda \mid \lambda \leq \mu \}.$$
Observe that for any $\lambda, \mu \in X$ the interval $\{\nu \in X \mid \lambda \leq \nu \leq \mu \}$ is always finite (possibly empty).

\subsection{The category $\mathcal O_p$}
Recall that $U_K$ has a triangular decomposition $U_K = U_K^- U_K^0 U_K^+$, cf. \cite{An22}, and set $B_K = U_K^0 U_K^+$. Using the above notation we may define $\mathcal O_p$ as follows.
\begin{defn} \label{O}
The category $\mathcal O_p$ is the full subcategory of the category of $U_K$-modules consisting of those $M$, which satisfy
\begin{enumerate}
\item $M$ is a weight module, i.e. $M = \bigoplus_{\lambda \in X} M_\lambda$ as a $U_K^0$-module.
\item $\dim_k (M_\lambda) < \infty$ for all $\lambda$, and there exist finitely many $\mu_1, \mu_2, \cdots , \mu_r \in X$ such that if $M_\lambda \neq 0$ then  $\lambda \in X(\leq \mu_1) \cup X(\leq \mu_2) \cup \cdots  \cup X(\leq \mu_r)$.
\end{enumerate}
\end{defn}
\begin{rem} \begin{enumerate}
\item This is equivalent to \cite{An22}, Definition 2.1 once one observes that condition (3) in that definition is contained in conditions (1) and (2) above. In fact, let $m \in M$ be non-zero. To prove that $\dim (B_K m) < \infty$ it is by (1) enough to verify this for $m \in M_\lambda$ for some $\lambda \in X$. However, $B_K m \subset \oplus_{\lambda \leq \mu} M_\mu$ and if $M_\mu \neq 0$ we know by (2) that $\mu \leq \mu_i$ for some $i$. Hence the non-zero terms in this direct sum only involves $\mu$'s belonging to the finitely many intervals $[\lambda, \mu_i], \; i = 1, 2, \cdots r$. By (2) all weight spaces $M_\mu$ for $M$ are finite dimensional, so we see that $\dim_K B_K m$ is indeed finite dimensional. 
\item In Definition 2.1 (2) we may always choose the $\mu_i$'s in different equivalences classes for $\Z R$ (cf. the arguments verifying Remark \ref{rem} below). In particular, the number $r$ can always be taken to be at most equal to the index of connection $|X:\Z R|$ for all $M$.
\end{enumerate}
\end{rem}

Observe that $\mathcal O_p$ is an abelian category. Moreover, it is also closed under taking (finite) tensor products. If $M, N \in \mathcal O_p$ have weights contained in $\cup_{i = 1}^r X({\leq \mu_i})$ and $\cup_{j = 1}^s X({\leq \nu_j})$, respectively, then the weights of $M \otimes_K N$ are contained in $ \cup_{i,j} X(\leq \mu_i + \nu_j)$.

Among the modules in $\mathcal O_p$ we have the finite dimensional modules. These form a subcategory $\mathcal C$ of $\mathcal O_p$, which identifies with the category of finite dimensional rational $G$-modules (studied in full details in \cite{RAG}).

\subsection{Formal characters}\label{characters}

Using the general concept ``formal characters'' introduced in \cite{DGK}, Section 3 we shall now define characters of the modules in $\mathcal O_p$. The same approach is also used in \cite{PF}, Section 2.5. Note that here and in the future we  write just characters (dropping the word ``formal'').

If $f: X \rightarrow \Z$ is a map then we set $\supp (f) = \{\lambda \in X \mid f(\lambda) \neq 0 \}$. We define
$$ \widehat{\Z[X]} = \{f: X \rightarrow \Z \mid \text { there exist } \mu_1, \cdots , \mu_r \in X \text { such that } \supp(f) \subset X(\leq \mu_1) \cup \cdots \cup X(\leq \mu_r) \}.$$

Let $M \in \mathcal O_p$. Then we define the character of $M$ to be the map $  \cha M : X \rightarrow \Z$ given by
$$ \cha M (\lambda) = \dim_K M_\lambda, \, \lambda \in X.$$
Note that the finiteness condition (2) in Definition \ref{O} ensures that $\cha M \in \widehat{\Z[X]}$. We have that $\supp(\cha M)$ is the set of weights of $M$.

Pointwise addition makes $\widehat{\Z[X]}$ into an abelian group. 
Clearly, taking characters is additive, i.e. if
$ 0 \to M_1 \to M \to M_2 \to 0$
is a short exact sequence in $\mathcal O_p$, then $\cha M = \cha M_1 + \cha M_2$.

We also have a product $\star$ on $\widehat{\Z[X]}$. It is defined by
$$ f \star g (\lambda) = \sum_{\mu \in X} f(\mu) g(\lambda - \mu), \; \lambda \in X.$$
Observe that when $f, g \in \widehat{\Z[X]}$ the sum on the right hand side is finite. Hence $f \star g$ is a well defined element of $\widehat{\Z[X]}$. In this way $\widehat{\Z[X]}$ becomes a commutative ring. Its $1$-element is $\cha K$ where $K$ denotes the trivial $1$-dimensional  $U_K$-module.
Note also that if $M_1, M_2 \in \mathcal O_p$ then we have
$$ \cha (M_1 \otimes M_2) = (\cha M_1) \star (\cha M_2). $$

For later use we also record the fact that $\widehat{\Z[X]}$ is closed under certain infinite sums. Namely, suppose $(f_i)_{i \in I}$ is a family of elements in $\widehat{\Z[X]}$ which satisfies
\begin{enumerate}
\item there exist $\mu_1, \cdots , \mu_r \in X$ such that $\supp (f_i) \subset X(\leq \mu_1) \cup \cdots \cup X(\leq \mu_r)$ for all $i \in I,$
\item for any $\lambda \in X$ we have $\lambda \in \supp (f_i)$ for at most finitely many $i \in I$. 
\end{enumerate}
If (1) and (2) are satisfied the family $(f_i)_{i \in I}$ is  called {\it summable}. In that case 
the (possibly infinite) sum $\sum_{i \in I} f_i$ equals  the map $f \in \widehat{\Z[X]}$ defined by 
$$ f(\lambda) = \sum_{i \in I} f_i(\lambda), \; \lambda \in X.$$

Let $\lambda \in X$. Define $e^\lambda$ to be the characteristic function on $\{\lambda\} \subset X$, i.e. $e^\lambda (\mu) = \delta_{\mu, \lambda}$. Then for any $f \in \widehat{\Z[X]}$ the family
$(f(\lambda) e^\lambda)_{\lambda \in X}$ is summable (and has sum $f$). In fact, this  family is summable if and only if $f \in \widehat{\Z[X]}$.

Let $r \in \Z_{\geq 0}$. Then for any map $f: X \to \Z$ we define $f^{(r)} : X \to \Z$ by 
$$ f^{(r)} (\lambda) = f(p^r \lambda), \; \lambda \in X.$$
Clearly, $\supp f^{(r)} = p^r \supp (f)$. In particular, we see that $f  \in \widehat{\Z[X]}$ if and only if $f^{(r)} \in \widehat{\Z[X]}$.

Recall from \cite{An22}, 3.6 that for each $M \in \mathcal O_p$ we have the Frobenius twists $M^{(r)} \in \mathcal O_p$, $ r \geq 0$. Clearly, the characters of these twisted modules are given by 
$$ \cha M^{(r)} = (\cha M)^{(r)}.$$
We shall often write $\cha^{(r)} M$ instead of $\cha (M)^{(r)}$.

\subsection{Verma modules  and simple modules in $\mathcal O_p$}

Recall that the Verma module $\Delta_p(\lambda)$ in $\mathcal O_p$ with highest weight $\lambda \in X$ is defined by
$$ \Delta_p(\lambda) = U_K \otimes_{B_K} \lambda,$$
where $\lambda$ is considered a $B_K$-module via the natural projection $B_K \to U_K^0$. 
It follows immediately from this definition that as a $U^-_K U^0_K$-module we have $\Delta_p(\lambda) \simeq U_K^- \otimes_K \lambda$. Hence we get from the PBW-basis of $U_K^-$ the following character formula (with $e^\mu$ as in Section \ref{characters})
\begin{equation} \label{ch-Verma}
\cha \Delta_p(\lambda) = \sum_{\mu \leq \lambda} P(\lambda - \mu) e^\mu,
\end{equation}
where $P$ is the Kostant partition function.

Each Verma module $\Delta_p(\lambda)$ has a unique simple quotient which we denote $L_p(\lambda)$. Then clearly the weights of $L_p(\lambda)$ are also $\leq \lambda$, and the multiplicity of $\lambda$ as a weight of $L_p(\lambda)$ is $1$. Up to isomorphisms the set $\{L_p(\lambda) \mid \lambda \in X \}$ is the family of all simple modules in $\mathcal O_p$.

\subsection{Steinberg's tensor product theorem and character formulas}
As in \cite{An22} we set
$$ X^+ = \{\lambda \in X \mid \langle \lambda, \alpha_i^\vee \rangle \geq 0 \text { for all } i\},$$
and for $r \geq 0$
$$X_r = \{\lambda \in X \mid 0 \leq \langle \lambda, \alpha_i^\vee \rangle < p^r \text { for all } i\}.$$
$X^+$ is the set of dominant weights and the elements of $X_r$ are called the $p^r$-restricted  weights (here we have deviated from the notation used in \cite{An22}).

For a fixed $r$ we can for each $\lambda \in X$ write $\lambda = \lambda^0 + p^r \lambda^1$ for unique $\lambda^0 \in X_r$ and $\lambda^1 \in X$. Note that $\lambda^0$ and $\lambda^1$ depend on $r$. In the formulas below it should be clear from the context which $r$ we are working with (special care should be taken in Theorem \ref{dom} below where we have abused notation further).

Steinberg's tensor product theorem, Theorem 3.6  in \cite{An22}, for $\mathcal O_p$ then says that we have an isomorphism
\begin{equation} \label{st} L_p(\lambda) \simeq L_p(\lambda^0) \otimes L_p(\lambda^1)^{(r)}.\end{equation}

Recall that $L_p(\lambda)$ is finite dimensional if and only if $\lambda \in X^+$. For such $\lambda$ we write $\lambda = \lambda^0 + p \lambda^1 + \cdots + p^s \lambda^s$ with $\lambda^i \in X_1$ for all $i = 1, \cdots , s$ and get  the following well known formula for the characters of finite dimensional simple modules in $\mathcal O_p$ (i.e for the  finite dimensional simple $G$-modules), cf. \cite{St}.

\begin{thm} \label{dom} Let $\lambda \in X^+$. Then with the above notation
$$ \cha L_p(\lambda) = \cha L_p(\lambda^0) \star \cha^{(1)} L_p(\lambda^1)\star \cdots \star  \cha^{(s)} L_p(\lambda^{s}
).$$
\end{thm}

 \begin{proof} Iterate the $s=1$ case of (\ref{st}) and use the formulas from Section \ref{characters}.
 \end{proof}
 
 Suppose next $\lambda \in X$ is arbitrary. Set 
 $$ Y = \{\omega \in X \mid \langle \omega, \alpha_i^\vee \rangle \in \{-1, 0\} \text { for all } i= 1, \cdots, n\}.$$
Then we have

\begin{thm} \label{gen} For $\lambda \in X$ we have for all  $r\geq \mi\{s \mid -p^s > \langle \lambda, \alpha_i^\vee \rangle > p^s \text{ for all } i\}$ that $\lambda = \lambda^0 + p^r \omega$ for unique $\lambda^0 \in X_r$ and $ \omega \in Y$, and  
$$ \cha L_p(\lambda) = \cha L_p(\lambda^0) \star \cha^{(r)} L_p(\omega).$$
\end{thm}

\begin{proof} This is an immediate consequence of (\ref{st}).
\end{proof} 

\begin{rem}
Let $\omega \in Y$. Then $(1-p^r)\omega \in X_r$ and hence Steinberg's tensor product theorem tells us that $L_p(\omega) \simeq L_p((1-p^r)\omega) \otimes L_p(\omega)^{(r)}$. Hence 
$ \cha L_p(\omega) = \cha L_p((1-p^r)\omega) \star \cha^{(r)} \cha L_p(\omega)$.
\end{rem}

Observe that if $\lambda$ is antidominant (i.e. $\langle \lambda, \alpha_i^\vee \rangle < 0$ for all $i= 1, \cdots , n$), then the $\omega \in Y$ appearing in Theorem \ref{gen} equals $- \rho$. According to \cite{An22}, Theorem 6.1 we have $L_p(-\rho) = \Delta_p(\-\rho)$. Hence in this case Theorem \ref{gen} gives the following formula.

\begin{cor} \label{anti} Let $\lambda \in X$ be antidominant and take $r$ so large that $p^r \rho + \lambda \in X^+$. Then we have 
$$ \cha L_p(\lambda) = \cha L_p(p^r\rho + \lambda) \star \cha^{(r)} L_p(-\rho) =  \cha L_p(p^r \rho + \lambda) \star (\sum_{\mu \geq 0} P( \mu) e^{- p^r(\rho + \mu)}).$$
\end{cor}

\begin{rem}
\begin{enumerate}
\item Assume that $\cha L_p(\nu)$ are known for all $\nu$ in the finite set $ X_1$ (this is equivalent to knowing all irreducible $G_1
T$ characters, $G_1$ being the Frobenius subgroup scheme of $G$ and $T$ is a maximal torus in $G$). Then Theorem \ref{dom}, respectively Corollary \ref{anti}, gives explicit formulas for how to obtain from this the characters of all simple modules in $\mathcal O_p$ with dominant, respectively antidominant, highest weights.
\item Assume that $\cha L_p(\nu)$ are known for all $\nu$ in the finite set $ X_1 \cup Y$. Then Theorem \ref{gen} gives explicit formulas for how to obtain from this the characters of all other simple modules in $\mathcal O_p$.
\item Recall that \cite{An22}, Corollary 3.10 gives recipes for determining all the remaining irreducible characters in $\mathcal O_p$ from  the assumption in (1) above, cf. also the periodicy theorem in \cite{PF}. However, in contrast to our formulas above,  these recipes work one weight at a time.
\item The recent  breakthrough, see  \cite{RW1}, determines for $p \leq 2h-1$ the characters $\cha L_p(\nu))$ for all $\nu \in X^+$
 in terms of the $p$-Kazhdan-Lusztig polynomials. Hence by (3) this leads to the determination of all simple characters in $\mathcal O_p$. 
 \end{enumerate}
 \end{rem}

\section{The linkage principle in $\mathcal O_p$}
In this section we shall show that $\mathcal O_p$ breaks up into a direct sum of subcategories corresponding to the orbits in $X$ of the affine Weyl group. This result is analogous to the splitting of the ordinary  category $\mathcal O$, and it is an even closer analogue of the situation in the category $\mathcal C$ of finite dimensional rational representations of $G$, see \cite{RAG}, II.6. As an immediate consequence of this result we get translation functors  in $\mathcal O_p$ (as in $\mathcal O$ and $\mathcal C$). 

\subsection{Local composition series}
Modules  in $\mathcal O_p$ do in general not have composition series in the usual sense, see e.g. the examples in \cite{An22}, Section 5. However, at any given weight we have a local composition series in the following sense:

\begin{prop} \label{local} (\cite{DGK}, Proposition 3.2)
Let $M \in \mathcal O_p$ and $\lambda \in X$. Then there exists a finite filtration  
$$ 0 = M_0 \subset M_1 \subset \cdots \subset M_r = M$$
of submodules of $M$ and a subset $J \subset \{1, 2, \cdots , r\}$ such that
\begin{enumerate}
\item if $i \in J$ then $M_i/M_{i-1} \simeq L_p(\lambda_i)$ for some $\lambda_i \in X$,
\item if $i \not \in J$ then $(M_i/M_{i-1})_\mu = 0$ for all $\mu \geq \lambda$.
\end{enumerate}
\end{prop}

A filtration like the one in this proposition is called a {\it local composition series} for $M$ at $\lambda$.

\begin{examplecounter} \label{ex1} Let $G = SL_2(K)$. In this case $X = \Z$. Take $r \geq 0$. Then the results in \cite{An22}, Section 5.2 give
\begin{enumerate}
\item The Verma module $\Delta_p(0)$ has a filtration $0 = M_0 \subset M_1 \subset \cdots \subset M_{r+2} = \Delta_p(0)$ with $M_i/M_{i-1}  = L_p(-2p^{i-1}), \; i= 1, 2, \cdots , r$ and $M_{r+1} = \rad \Delta_p(0).$ 
This is a local composition series for $\Delta_p(0)$ at $-2p^r + 2$. The associated subset consisting of all $i$ for which $M_i/M_{i-1}$ is simple is $J = \{1, 2, \cdots r, r+2 \}.$
\item The Verma module $\Delta_p(-2)$ has a filtration $0 = N_0 \subset N_1 \subset \cdots \subset N_{r+1} = \Delta_p(2)$ with $N_1 = \Delta_p(-2)^{(r)}, \; N_i/N_{i-1}  = L_p(-2p^{r+1-i}), \; i= 2, \cdots , r+1$. This is a local decomposition series for $\Delta_p(-2)$ at $-2p^r + 2$ with associated subset $J = \{2, \cdots , r+1 \}.$
\end{enumerate}
\end{examplecounter}

\subsection{Composition factor multiplicities} \label{Comp. Fact.}
Recall from Section \ref{characters} the concept of summability of certain families of elements in $\widehat {\Z[X]}$. Using this it follows as in \cite{DGK}, Section 3 that for each $M \in \mathcal O_p$ there exist unique non-negative numbers $([M:L_p(\mu)])_{\mu \in X}$ such that the family  $([M:L_p(\mu)] \cha L_p(\mu))_{\mu \in X}$ is summable and the following formula holds in $\widehat {\Z[X]}$
$$ \cha M = \sum_{\mu \in X} [M:L_p(\mu)] \cha L_p(\mu).$$
Suppose $\mu_1, \cdots , \mu_r$ are elements of $X$ such that all weights of $M$ are contained in the set $X(\leq \mu_1) \cup \cdots \cup X(\leq \mu_r)$. Then we have 
$$ [M : L_p(\mu)] > 0 \text { implies } \mu \in X(\leq \mu_1) \cup \cdots \cup X(\leq \mu_r).$$
We shall call $[M:L_p(\mu)]$ the composition factor multiplicity of $L_p(\mu)$ in $M$ and say that $L_p(\mu)$ is a composition factor of $M$ whenever this number is non-zero. Note that this is slightly misleading as $M$ may not have a composition series in the usual sense, cf. Example \ref{ex1} above. In \cite{DGK} the term `simple component' is used instead of `composition factor'. 

Composition factors clearly have the following property.
\begin{rem}
For all $M \in \mathcal O_p$ the simple module $L_p(\lambda)$ is a composition factor of $M$ if and only if $L_p(\lambda)$ is a subquotient of $M$.
\end{rem}

Let now $\lambda, \mu \in X$ and consider $M = \Delta_p(\lambda)$.  Set $a_{\mu, \lambda} = [\Delta_p(\lambda):L_p(\mu)] \in \Z_{\geq 0}$. Then $a_{\mu, \lambda} = 0$ unless $\mu \leq \lambda$, $a_{\lambda, \lambda} =1$, and we have 
$$ \cha \Delta_p(\lambda) = \sum_\mu a_{\mu, \lambda} \cha L_p(\mu).$$
Fix now $\eta \in X$ and consider the (infinite) matrix $(a_{\mu, \lambda})_{\mu, \lambda \in X(\leq \eta)}$. Then by the above observations this is an upper triangular matrix (wrt the order $\leq$ on $(X(\leq \eta)$) with $1$'s on the diagonal. Hence it is invertible and we denote its inverse by $(c_{\mu, \lambda})_{\mu, \lambda \in X(\leq \eta)}$. This is also an upper triangular matrix with $1$'s on the diagonal and with all entries in $\Z$. If we take $\eta \geq \lambda$ then we have
$$  \cha L_p(\lambda) = \sum_\mu c_{\mu, \lambda} \cha \Delta_p(\mu).$$
These equations make explicit how we can obtain the irreducible characters in $\mathcal O_p$ once we know the composition factor multiplicities of Verma modules - and vice versa. Note that the integers 
$a_{\mu, \lambda}$ and $c_{\mu, \lambda}$ are independent of $\eta$ (as long as $\lambda \leq \eta$).

\subsection{The strong linkage principle for Verma modules}

Recall the following result from \cite{An22}, Theorem 4.5
\begin{thm} \label{filt} Let $\lambda \in X$ and $r \in Z_{\geq 0}$. Then $\Delta_p(\lambda)$ has a finite filtration with subquotients isomorphic to $L_p(\mu^0) \otimes \Delta_p(\mu^1)^{(r)}$ for some $\mu \in X$. Moreover, if  $L_p(\mu^0) \otimes \Delta_p(\mu^1)^{(r)}$ occurs at least once as such a subquotient, then $\mu$ is strongly linked to $\lambda$.
\end{thm}

We deduce the following corollary
\begin{cor} \label{SLP} Let $\lambda, \mu \in X$. If $[\Delta_p(\lambda):L_p(\mu)] \neq 0$ then $\mu $ is strongly linked to $\lambda$.
\end{cor}

\begin{proof} Let $0 = F_0 \subset F_1 \subset \cdots \subset F_s = \Delta_p(\lambda)$ be a filtration as in the $r=1$ case of Theorem \ref{filt} with $F_i/F_{i-1} \simeq L_p(\mu_i^0) \otimes \Delta_p(\mu_i^1)^{(1)}, i = 1, \cdots , s$. 
Then the additivity  of $\cha $ gives
$$ \cha \Delta_p(\lambda) = \sum_{i=1}^s \cha (L_p(\mu_i^0) \otimes \Delta_p(\mu_i^1)^{(1)}).$$
As each $\mu_i$ is strongly linked to $\lambda$ (by Theorem \ref{filt}) we are done if we prove
$$ [L_p(\mu_i^0) \otimes \Delta_p(\mu_i^1)^{(1)} :L_p(\mu)] \neq 0 \text { implies } \mu \text { is strongly linked to } \mu_i.$$
However, the composition factors of $L_p(\mu_i^0) \otimes \Delta_p(\mu_i^1)^{(1)}$ have the form  $L_p(\mu_i^0) \otimes L_p(\nu)^{(1)}$ where $L_p(\nu)$ is a composition factor of $\Delta_p(\mu_i^1)$. The last condition implies $\nu \leq \mu_i^1$ so that we have $\mu_i^0 + p\nu$ is strongly linked to $\mu_i^0 + p\mu_i^1 = \mu_i$.
\end{proof} 

\subsection{The duality functor}
Recall that just as in ordinary category $\mathcal O$ we have a duality functor $D: \mathcal O_p \to \mathcal O_p$, see \cite{An22}, Remark 4.11. As $D$ preserves weight spaces we have
$$ \cha DM = \cha M \text { for all } M \in \mathcal O_p.$$
and hence also 
$$ [DM: L_p(\lambda)] = [M : L_p(\lambda)] \text { for all } \lambda \in X.$$

As in \cite{An22} we write $\nabla_p(\lambda)$ instead of $D\Delta_p(\lambda)$. By the above observations we see that the strong linkage principle holds just as well for the dual Verma modules.

For later use we record the following result, cf. \cite{An22}, Lemma 8.7(2),
\begin{lem} \label{Delta,nabla}
For all $\lambda, \mu \in X$ we have $\Ext^1_{\mathcal O_p} (\Delta_p(\lambda), \nabla_p(\mu)) = 0$. 
\end{lem}

\subsection{The decomposition of  $\mathcal O_p$ into linkage classes} \label{decomp}

Denote by $W_p$ the affine Weyl group corresponding to $R$. This group acts on $X$ by the dot action $w \cdot \lambda = w(\lambda + \rho) - \rho$, see \cite{RAG}, II.6.1. Clearly,  if $\mu$ is strongly linked to $\lambda$, then we have in particular $\mu \in W_p \cdot \lambda$.

Let $C = \{\nu \in X^+ | \langle \nu + \rho, \beta^\vee \rangle < p \text { for all } \beta \in R^+\}$ be the bottom alcove in $X^+$ and set
$$ \bar C = \{\nu \in X | 0 \leq \langle \nu + \rho, \beta^\vee \rangle \leq p \text { for all } \beta \in R^+\}.$$
Then $\bar C$ is a fundamental domain for the action of $W_p$ on $X$. If $\nu \in \bar C$ we define $\mathcal O_p(\nu)$ to be the full subcategory of $\mathcal O_p$ given by
$$ \mathcal O_p(\nu) = \{M \in \mathcal O_p | [M:L_p(\mu)] = 0 \text { unless } \mu \in W_p\cdot \nu \}.$$

 \begin{thm} \label{blocks}
 Let $M \in \mathcal O_p$. Then there exists a unique set $(M(\nu))_{\nu \in \bar C}$ of submodules of $M$ with 
 $$ M(\nu) \in \mathcal O_p(\nu) \text { for all } \nu \in \bar C \text { and } M = \bigoplus_{\nu \in \bar C} M(\nu).$$
\end{thm}

The proof that we present below is obtained by specializing/simplifying the arguments in \cite{DGK}, Section 4. First we prove
\begin{prop} \label{Hom,Ext}
Let $\nu \in \bar C$ and $\mu \in W_p\cdot \nu$. If $N \in \mathcal O_p$ satisfies $[N:L_p(w \cdot \nu)] = 0$ for all $w \in W_p$ then
\begin{enumerate}
\item $\Hom_{\mathcal O_p} (\Delta_p(\mu), N) = 0$, 
\item $\Ext^1_{\mathcal O_p} (\Delta_p(\mu), N) = 0$ 
\end{enumerate}
\end{prop}

\begin{proof}
$(1)$ follows from the observation that if $\Hom_{\mathcal O_p} (\Delta_p(\mu), N) \neq 0$ then $L_p(\mu)$ is a subquotient of $N$ contradicting our assumption on $N$. 

To prove $(2)$ we begin by noting that
\begin{equation} \label{Delta,L}
\Ext^1_{\mathcal O_p} (\Delta_p(\mu), L_p(\lambda)) = 0 \text { for all } \lambda  \text { for which } [N:L_p(\lambda)] \neq 0.
\end{equation}
In fact,  we get from Lemma \ref{Delta,nabla} a surjection $\Hom_{\mathcal O_p}(\Delta_p(\mu), \nabla_p(\lambda)/L_p(\lambda)) \to \Ext^1_{\mathcal O_p} (\Delta_p(\mu), L_p(\lambda))$. Then (\ref{Delta,L}) follows from the first part of our proposition combined with (the ``dual'' of)  Corollary \ref{SLP}.

Take now a local composition series
$$ 0 = N_0 \subset N_1 \subset \cdots \subset N_r = N$$
of $N$ at $\mu$, cf. Proposition \ref{local} and let $J$ be the subset of $\{1, 2, \cdots , r\}$ as in loc. cit. By (\ref{Delta,L}) we have $\Ext^1_{\mathcal O_p} (\Delta_p(\mu), N_i/N_{i-1}) = 0$ for all $i \in J$. However, we also have $\Ext^1_{\mathcal O_p} (\Delta_p(\mu), N_i/N_{i-1}) = 0$ for $i \notin J$ because by the universal property of Verma modules we have $\Ext^1_{\mathcal O_p} (\Delta_p(\mu), N') = 0$ for all $N' \in \mathcal O_p$ which have no weights bigger than $\mu$.
\end{proof}
\begin{cor}\label{vanishing} Let $\nu$ and $N$ be as in Proposition \ref{Hom,Ext}. Then
\begin{enumerate}
\item $\Hom_{\mathcal O_p}(M, N) = 0$ for all $M \in \mathcal O_p(\nu)$,
\item if $M \in \mathcal O_p(\nu)$ is a highest weight module then $\Ext^1_{\mathcal O_p} (M, N) = 0$.
\end{enumerate}
\end{cor} 
\begin{proof}
To prove $(1)$ we consider a non-zero element  $f \in \Hom_{\mathcal O_p} (M, N)$. Pick a maximal weight $\lambda$ in $\im(f)$. Then $0 \neq \Hom_{\mathcal O_p} (\Delta_p(\lambda), \im(f)) \subset \Hom_{\mathcal O_p}(\Delta_p(\lambda), N)$. Note that $L_p(\lambda)$ is a subquotient of $\im(f)$ and hence also a subquotient of $M$. As $M \in \mathcal O_p(\nu)$ it follows that $\lambda \in W_p\cdot \nu$. Hence we have a contradiction to Proposition \ref{Hom,Ext} (1).

 We now turn to $(2)$. Denote the highest weight of $M$ by $\mu \in W_p\cdot \nu$ and let $R$ denote the kernel of the surjection $\Delta_p(\mu) \to M$. Then Proposition \ref{Hom,Ext} (2) gives us a surjection  $\Hom_{\mathcal O_p} (R, N) \to \Ext^1_{\mathcal O_p} (M, N)$. However, because of Corollary \ref{SLP} we have $R \in \mathcal O_p(\nu)$ and therefore $\Hom_{\mathcal O_p} (R, N) = 0$ by Proposition \ref{Hom,Ext} (1). 
\end{proof}
With this we are now ready to prove Theorem \ref{blocks}:
\begin{proof}
Let $M \in \mathcal O_p$. We first construct a ``highest weight module'' filtration of $M$ as follows. Choose first  a maximal weight $\lambda_1$ among the weights of $M$ together with a non-zero vector $v_1 \in M_{\lambda_1}$. Set $M_1 = U_K v_1$ and choose then a maximal weight $\lambda_2$ among the weights of $M/M_1$ together with a non-zero vector $v_2 \in (M/M_1)_{\lambda_2}$. Define $M_2$ as the preimage in $M$ of the submodule $U_K v_2 \subset M/M_1$. Continue this procedure to obtain a (possibly infinite) filtration
$$0 = M_0 \subset M_1 \subset M_2 \subset M_3 \subset \cdots \subset M,$$
which satisfies (compare \cite{DGK}, Proposition 3.1)
\begin{enumerate}
\item $M = \cup_{i} M_i$,
\item $M_i/M_{i-1}$ is a highest weight module with highest weight $\lambda_i$,
\item if $\lambda_i > \lambda_j$ then $i < j$,
\item for each $\lambda \in X$ there exists $r \in \Z_{\geq 0}$ such that $(M_r)_\lambda = M_\lambda$.
\end{enumerate}
We claim that the theorem holds for each term $M_i$ in this filtration. This is certainly true for the highest weight module $M_1$. In fact, Corollary \ref{SLP} ensures that $M_i/M_{i-1} \in \mathcal O_p(\nu_i)$ where $\nu_i$ is the unique element in $\bar C \cap W_p\cdot\lambda_i$. Suppose the claim is true for $i-1$, i.e. that there are submodules $M_{i-1}(\nu) \in \mathcal O_p(\nu)$ of $M_{i-1}$ with $M_{i-1} = \oplus_{\nu \in \bar C} M_{i-1}(\nu)$.
Then it follows from Corollary \ref{vanishing} that $M_i = ( \oplus_{\nu \in \bar C\setminus\{\nu_i \}} M_{i-1}(\nu)) \oplus M_i(\nu_i)$, where $M_i(\nu_i)$ is the extension $0 \to M_{i-1}(\nu_i) \to M_i(\nu_i) \to M_i/M_{i-1} \to 0$. Note that each $M_i(\nu)$ is the maximal submodule of $M_i$ belonging to $\mathcal O_p(\nu)$.

Set now $M(\nu) = \cup_i M_i(\nu), \; \nu \in \bar C$. It then follows that $M = \oplus_{\nu \in \bar C} M(\nu)$ as desired.
\end{proof}

\begin{rem}\label{rem} Let $\nu \in \bar C$. The component $\mathcal O_p(\nu)$ may then be described as the subcategory of all $U_k$-modules $M$ which satisfy
\begin{enumerate}
\item $M$ is a weight module,
\item $\dim M_\lambda < \infty$ for all $\lambda \in X$, and the set of weights of $M$ are bounded from above by some $\mu \in \nu + \Z R$,
\item $[M: L_p(\lambda)] = 0$ unless $\lambda \in W_p\cdot \nu$
\end{enumerate}
It is clear that if $M$ satisfies these conditions then $M \in \mathcal O_p(\nu)$. The only thing we have to prove to check the reverse statement is that the weights of $M$ are bounded by a single element in $\nu  + \Z R$. However, if $\lambda$ is a weight of $M$ then by (3) $\lambda$ is also a weight of $L_p(w \cdot \nu)$ for some $w \in W_p$. Then $\lambda \leq w \cdot\nu$, and since $w\cdot \nu \in \nu + \Z R$ we also get $\lambda \in \nu + \Z R$. Thus $\supp M \cap X(\leq \mu) = \emptyset$ unless $\mu \in \nu + \Z R$. 
As $M \in \mathcal O_p$ there exist finitely many $\mu_i \in X$ such that $\supp (\cha M) \subset \cup_i X(\leq \mu_i)$. By our observations we may assume that $\mu_i \in \nu + \Z R$ for all $i$. But if $\mu_1, \mu_2 \in \nu + \Z R$ we may write $\mu_2 = \mu_1 + \sum_j n_j \alpha_j$ for some $n_j \in \Z$. We note that then $\mu_1, \mu_2 \leq \mu_1 + \sum_j |n_j|\alpha_j$. It follows that we can find $\mu \in \nu + \Z R$ such that $\mu_i \leq \mu$ for all $i$. Hence this $\mu$ is an upper bound for the set of weights of $M$.

\end{rem}

\subsection{Translation functors} \label{trans}

The decomposition of $\mathcal O_p$ in Theorem \ref{blocks} allows us to define translation functors like those in category $\mathcal O$ and category $\mathcal C$, cf. \cite{RAG}, II.7.

Below we give the main definitions and state the main properties of translation functors. For proofs we refer to \cite{RAG}, Chapter II.7.

Let $C$ and $\bar C$ be as in Section \ref{decomp}. If $\lambda \in \bar C$ and $M \in \mathcal O_p$ then the summand of $M$ belonging to $ \mathcal O_p(\lambda)$ (when decomposing $M$ as in Theorem \ref{blocks}) is the maximal submodule $pr_\lambda (M)$ of $M$ with composition factors belonging to $\{L_p(w \cdot \lambda) \mid w \in W_p \}$. The functor $\mathcal O_p \to \mathcal O_p(\lambda)$ given by the assignment $M \mapsto pr_\lambda(M)$ on objects (and by the obvious restriction map on morphisms, cf. Corollary \ref{vanishing}) is denoted $pr_\lambda$.

\begin{defn} Suppose $\lambda, \mu \in \bar C$. Then the translation functor $T_\mu^\lambda : \mathcal O_p(\mu) \to \mathcal O_p(\lambda)$ is given by
$$ T_\mu^\lambda (M) =  pr_\lambda(M \otimes L_p((\lambda - \mu)^+))), \; M \in \mathcal O_p(\mu),$$
where $(\lambda - \mu)^+$ is the unique weight belonging to $X^+ \cap W (\lambda - \mu)$.

\end{defn}

\begin{rem} In \cite{RAG}, Chapter II.7 the corresponding translation functor is considered as the endofunctor on $\mathcal C$ given (in our notation) as the composite $T_\mu^\lambda \circ pr_\mu$. 
\end{rem}

Among the main properties of translation functors we have

\begin{enumerate}
\item $T_\mu^\lambda$ is an exact functor.
\item $T_\mu^\lambda$ is left and right adjoint to $T_\lambda^\mu$.
\vskip .5 cm
Suppose $\mu \in C$ and $w \in W_p$. Then
\vskip .5 cm
\item $T_\mu^\lambda(\Delta_p(w \cdot \mu)) \simeq \Delta_p(w \cdot \lambda)$ and $T_\mu^\lambda(\nabla_p(w \cdot \mu)) \simeq \nabla_p(w \cdot \lambda)$.
\item $T_\mu^\lambda(L_p(w \cdot \mu)) \simeq \begin{cases} {L_p(w \cdot \lambda) \text { if } w\cdot \lambda \in \widehat{w \cdot C},}\\ {0  \text { otherwise.}} \end{cases}$
\end{enumerate}
Here $\widehat{w \cdot C}$ denotes the upper closure of $w\cdot C$, see \cite{RAG}, II.6.2.

\begin{rem} \label{sing} There are more general versions of (3) and (4) applying to arbitrary facets of $C$, see \cite{RAG}, II.7.14
\end{rem}

We shall now show that the above properties of translation functors allow us when $p \geq h$ to reduce the problem of determining the irreducible characters in $\mathcal O_p$ to solving the same problem in $\mathcal O_p(0)$. When $p \geq h$ we have $0 \in C$. The linkage principle (Corollary \ref{SLP}) ensures then that if $w \in W_p$ then $[\Delta_p(w \cdot 0): L_p(\mu)] = 0$ unless $\mu = y \cdot 0$ for some $y \in W_p$. This means that we have $\cha \Delta_p(w \cdot 0) = \sum_{y \in W_p} a_{y,w} \cha L_p(y \cdot 0)$, where $a_{y,w} = [\Delta_p(w \cdot 0): L_p(y \cdot 0)]$. We note that $a_{y,w} = 0$ unless $y \cdot 0 \leq w \cdot 0$, and $a_{w,w} = 1$. Arguing as in Section \ref{Comp. Fact.} we can therefore write
$$\cha L_p (w \cdot 0) = \sum_{\{y \in W_p \mid y \cdot 0 \leq w \cdot 0 \}} c_{y,w} \cha \Delta_p(y \cdot 0),$$
where $(c_{y,w})_{y,w}$ is the inverse matrix of  $(a_{y,w})_{y,w}$, see Section \ref{Comp. Fact.}.

Using the above notation we then have

\begin{prop} \label{reg} Assume $ p \geq h$ and suppose $\lambda \in \bar C$ has $w \cdot \lambda \in \widehat {w \cdot C}$. Then  $\cha L_p (w \cdot \lambda) = \sum_y c_{y,w} \cha \Delta_p(y \cdot \lambda)$.
\end{prop}

\begin{proof} This follows by the exactness of the translation functor $T_0^\lambda$ combined with its properties (3) and (4) above.
\end{proof}

\begin{rem} \begin{enumerate}
\item Note that when we apply Proposition \ref{reg} to the set of restricted weights we get that to determine the irreducible characters for all such weights it is enough to know the set $\{ \cha L_p(w\cdot 0) | w\cdot 0 \in X_1\}$. This is (for $R$ irreducible) a set of cardinality equal to $|W|/f$, see \cite{Hu} 4.9, where $f$ is the index of connection for $R$. Note that this number is (for $p \geq h$) much less than the number of weights in $X_1$ (which is $p^n$).
\item If $p < h$ then the more general versions referred to in Remark \ref{sing} gives a similar reduction to a single weight in each facet contained in $X_1$. 
\item The reductions in (1) and (2) then determine all irreducible characters in $\mathcal O_p$ via the Steinberg tensor product theorem, see Section 2. This generalizes the well known reductions in $\mathcal C$  (based on the classical Steinberg theorem for finite dimensional simple $G$-modules and the translation principles in $\mathcal C$).  
\end{enumerate}
\end{rem}

\section{Twisting functors} \label{twist0}
In this section we shall show that the twisting functors on $\mathcal O$, cf. \cite{Ark}, \cite{AL}, \cite{AS},  carry over to $\mathcal O_p$. In fact, we shall see that
twisting functors even have $\Z$-versions. This will be important for our applications in the next section.
\subsection{Twisting functors in characteristic $0$}
We start by briefly recalling the definition and some key properties of twisting functors on $\mathcal O$.

Consider the complex semisimple algebra $\mathfrak g$ from the introduction. We can write $\mathfrak g = \mathfrak n^- \oplus \mathfrak h \oplus \mathfrak n^+$,
where $\mathfrak h$ is a Cartan subalgebra, and $\mathfrak n^+$, respectively $\mathfrak n^-$, is the nilpotent Lie subalgebra corresponding to $R^+$, respectively $-R^+$. 
So we have $\mathfrak n^+= \bigoplus_{\beta \in R^+} \mathfrak g_\beta$, respectively $\mathfrak n^-= \bigoplus_{\beta \in R^+} \mathfrak g_{-\beta}$. We have correspondingly for the universal enveloping algebra $U = U(\mathfrak g) = U^-U^0U^+$ where $U^- = U(\mathfrak n^-), U^0 = U(\mathfrak h)$, and $U^+ =  U(\mathfrak n^+)$. These algebras are $X$-graded with finite dimensional graded pieces.  In particular we have
$U^- = \bigoplus_{\lambda \leq 0} U^-_\lambda$ where $\dim_\C U^-_{\lambda} = P(- \lambda)$.

Let $W$ denote the Weyl group for $R$. Then $W$ acts naturally on $\mathfrak g$. For each $w \in W$ we denote by  $\mathfrak n_w$ the Lie subalgebra of $\mathfrak n^-$ given by
$$ \mathfrak n_w = \mathfrak n^- \cap w^{-1}(\mathfrak n^+) = \bigoplus_{\{\beta \in R^+\mid w(\beta) < 0\}} \mathfrak g_{- \beta}.$$
We then set $N_w = U(\mathfrak n_w)$, the (complex) universal enveloping of $\mathfrak n_w$. This is a graded subalgebra of  $U^-$ and we have  $N_w = \bigoplus_{\{\lambda \leq 0 \mid w(\lambda) \geq 0\}} U^-_{\lambda}$. 

Let now $N_w^*$ denote the restricted dual of $N_w$, i.e. $N_w^* = \bigoplus_{\lambda \leq 0} ((N_w)_{\lambda})^*$. Then right, respectively left, multiplication by $N_w$ on itself give a left, respectively right $N_w$-module structure on $N_w^*$. We define the semiregular $(U,N_w)$-bimodule $S_w$ by
$$ S_w = U \otimes_{N_w} N_w^*.$$
The left $U$-structure is given by left multiplication on $U$ and the right $N_w$-structure comes from the right $N_w$-module structure on $N_w^*$.

It turns out that $S_w$ is also a right $U$-module and that as such we have $S_w \simeq N_w^* \otimes_{N_w} U$ with the right $U$-module structure on $N_w^* \otimes_{N_w} U$ coming from right multiplication on $U$. This allows us to define the twisting functor $T_w$ as follows.
\begin{defn} Let $w \in W$. Then $T_w$ is the endofunctor on the category of $U$-modules given by 
$$ T_w(M) =\; ^{\varphi_w}(S_w \otimes_U M),$$
where $\varphi_w$ denotes the twist by the automorphism on $U$ induced by the action of $w$ on $\mathfrak g$.
\end{defn}
Clearly, $T_w$ is right exact. The reason for the twist by $\varphi_w$ in this definition is that this makes the restriction of  $T_w$ to $\mathcal O$ into an endofunctor on the category $\mathcal O$ with the following properties, see  \cite{AL}, Section 6.2:

\begin{prop} \label{twist-prop} Let $w \in W$. Then
\begin{enumerate}
\item $T_w$ restricts to an endofunctor on $\mathcal O $.
\item $T_{ws} = T_w \circ T_s$ for all simple reflections $s$ with $ws > w$.
\item $\cha T_w \Delta(\lambda) = \cha \Delta (w \cdot \lambda)$ for all $\lambda \in X$.
\item $T_{w_0} \Delta(\lambda) \simeq \nabla(w_0 \cdot \lambda)$ for all $\lambda \in X$.
\end{enumerate}
\end{prop}
Here the order relation on $W$ appearing in (2) is the Bruhat order, and in (4) $w_0$ is as usual the longest element in $W$.

\subsection{Twisting functors over $\Z$} 
In this section we shall check that the construction of twisting functors has a version over $\Z$.

Choose a Chevalley basis for $\mathfrak g$. This is a basis $\{f_\beta, h_i, e_\beta \mid \beta \in R^+, i \in \{1, 2, \cdots , n\}\}$ for $\mathfrak g$ with particular nice properties. For instance, it makes 
$$ \mathfrak g_\Z = \spa_\Z\{f_\beta, h_i, e_\beta \mid \beta \in R^+, i \in \{1, 2, \cdots , r\}\}$$
into a Lie algebra over $\Z$. Among the relations we have $[e_{\alpha_i},f_{\alpha_i}] = h_i, i= 1, 2, \cdots, n$ (with $\{\alpha_1, \cdots, \alpha_n\}$ denoting the simple roots in $R^+$). Then the Kostant $\Z$-form $U_\Z \subset U$ is the $\Z$-subalgebra generated by the divided powers of $e_{\alpha_i}, f_{\alpha_i}, i = 1, 2, \cdots , n$. We have a PBW-type $\Z$-basis of $U_\Z$ consisting of 
products (in appropriate order) of divided powers of the basis-elements $f_\beta$ and $e_\beta$ together with the elements $\binom{h_i}{s} \in U^0_\Z = U_\Z(\mathfrak h)$.

In analogy with Section \ref{twist0}  for each $w \in W$ 
we then define $N_{w, \Z}$ as the graded $\Z$-subalgebra of $U_\Z^-$ generated by $\{f_\beta^{(s)} \mid \beta \in R^+, w(\beta) < 0, s \geq 0\}$. If $w = s_{i_r} \cdots s_{i_1}$ is a reduced expression of $w$ ($s_j$ being the reflection in $W$ associated with the simple root $\alpha_j$) 
and $\beta_j = s_{i_1} \cdots s_{i_{j-1}}(\alpha_{i_j})$, then $\{\beta_1, \cdots , \beta_r\} = \{\beta \in R^+ \mid w(\beta) < 0\}$ and if $\lambda \leq 0$ then the $\lambda$-graded piece of $N_{w,\Z}$ is a free $\Z$-module with basis $\{f_{\beta_r}^{(a_r)} \cdots f_{\beta_2}^{(a_2)}f_{\beta_1}^{(a_1)} \mid a_j \in \Z_{\geq 0}, \sum _{j=1}^r a_j \beta_j = - \lambda\}$. 

The above means that $N_\Z^* := \bigoplus_{\lambda > 0, \; w(\lambda) < 0} (N_{w,\Z})_{(-\lambda)})^*$  is also a free $\Z$-module. It is moreover, in analogy with $N_w^*$ above, a left and right $N_{w, \Z}$-module, so that we can define the semiregular  $(U_\Z, N_{w, \Z})$-bimodule by
$$ S_{w, \Z} = U_\Z \otimes_{N_{w, \Z}} (N_{w, \Z})^*.$$
In analogy with the situation in Section 4.1 we get that $S_{w, \Z}$ is also a right $U_\Z$-module and as such isomorphic to  $N^*_{w,\Z} \otimes_{N_{w, \Z}} U_\Z$. 
So this leads to the twisting functor $T_{w, \Z}$ on the category of $U_\Z$-modules given by
$$ T_{w, \Z} M = \; ^{\varphi_w}(S_{w, \Z} \otimes_{U_\Z} M),$$
where $\varphi_w$ now denotes the restriction of $\varphi_w$ to $\mathfrak g_\Z$. 

Consider the subcategory $\mathcal O_\Z$ consisting of $U_\Z$-modules $M$ which satisfy
\begin{enumerate}
\item $M$ has an $X$-grading, i.e. $M = \bigoplus_{\lambda \in X} M_\lambda$ as a $U_Z^0$-module, where $\binom{h_i}{s} m = \binom {\langle \lambda, \alpha_i^\vee\rangle}{s} m, \; m \in M_\lambda, i = 1, ,\cdots ,n, s \geq 0$.
\item All $M_\lambda$ are free over $\Z$ of finite rank and there exist $\mu_1, \cdots \mu_r \in X$ such that $M_\lambda = 0$ unless $\lambda \in X(\leq \mu_1) \cup \cdots \cup X(\leq \mu_r)$.
\end{enumerate}

Note that each $M \in \mathcal O_\Z$ has a well defined character $\cha M \in \widehat {\Z[X]}$ given by
$$ \cha M (\lambda) = \sum_{\mu \in X} rk(M_\lambda).$$
Clearly, $\mathcal O_\Z$ is not closed under taking quotients, because these may have torsion. In the next section we shall need to consider some purely torsion $\Z$-modules which occur as quotients of twisted Verma modules. For use in this context we make the following definition
\begin{defn} Let $N$ be a $U_\Z$-module which is $X$-graded just as the objects in $\mathcal O_\Z$ but with each graded piece $N_\lambda$ being a finite $\Z$-module. We assume that the set of weights of $N$ are bounded just as in (2) above. Then we define $\cha_p^t N \in \widehat{\Z[X]}$ by
$$ \cha_p^t N: \lambda \mapsto \nu_p(|N_\lambda|),$$
where $\nu_p (m) $ denotes the highest power of $p$ dividing $m \in \Z$.
\end{defn}
Clearly, $\cha_p^t$ is additive on exact sequences whose terms all have the same properties as $N$.
\vskip .5 cm
Let $\lambda \in X$ and set 
$$ \Delta_\Z(\lambda) = U_\Z \otimes_{B_\Z} \lambda.$$
We call this the Verma module over $\Z$ with highest weight $\lambda$. Clearly, $\Delta_\Z(\lambda) \in \mathcal O_\Z$ for all $\lambda$. We have also dual Verma modules over $\Z$ which we denote $\nabla_\Z(\lambda)$.

A direct analogue of Proposition \ref{twist-prop} is then 
\begin{prop} \label{twistZ} Let $w \in W$. Then
\begin{enumerate}
\item $T_{w, \Z}$ restricts to an endofunctor on $\mathcal O_\Z $.
\item $T_{ws, \Z} = T_{w, \Z} \circ T_{s, \Z}$ for all simple reflections $s$ with $ws > w$.
\item $\cha T_{w, \Z} \Delta_\Z(\lambda) = \cha \Delta_\Z (w \cdot \lambda)$ for all $\lambda \in X$.
\item $T_{w_0, \Z} \Delta_\Z(\lambda) \simeq \nabla_\Z(w_0 \cdot \lambda)$ for all $\lambda \in X$.
\end{enumerate}
\end{prop}

\subsection{Twisting functors in characteristic $p$}

It is straightforward to pass from the integral twisting functor $T_{w, \Z}$ to its analogue on the category of $U_K$-modules. In fact, we have by definition that $U_K = U_\Z \otimes_\Z K$. Moreover, if we for $w \in W$ define $S_{w, K}$ in analogy with the procedure in Section \ref{twist0},  then we also have
$$ S_{w, K} = S_{w, \Z} \otimes_\Z K.$$
The twisting functor $T_{w, K}$ on the category of $U_K$-modules is then given by
$$ T_{w, K} M = \; ^{\varphi_w}(S_{w; K} \otimes_{U_K} M).$$
Again this is a right exact functor, it restricts to an endofunctor on $\mathcal O_p$, we have $T_{ws, K} = T_{w, K} \circ T_{s, K}$ for all simple reflections $s$ for which $ws > w$,  $\cha T_{w, K} \Delta_p(\lambda) = \cha \Delta_p(w \cdot \lambda)$, and $T_{w_0 ,K} \Delta_p(\lambda) \simeq \nabla_p(\lambda)$ for all $\lambda \in X$. In addition, we have

\begin{prop} The twisting functors on $\mathcal O_p$ preserve each of the subcategories $\mathcal O_p(\nu), \; \nu \in \bar C$.
\end{prop}

\begin{proof} Let $\nu \in \bar C$. If $w \in W$ and $\lambda \in W_p \cdot \nu$ then we have (as observed above) $\cha T_{w,K} \Delta_p(\lambda) = \cha \Delta_p(w\cdot \lambda)$. This implies that $T_{w,K} \Delta_p(\lambda) \in \mathcal O_p$. The right exactness of $T_{w, K}$ then gives that also $T_{w, K} L_p(\lambda) \in \mathcal O_p(\nu)$. The proposition follows.
\end{proof}

We say that $M \in \mathcal O_p$ has a $\Z$-form if there exists $M_\Z \in \mathcal O_\Z$ with $M_\Z \otimes_\Z K = M$. If $M$ has a $\Z$-form $M_\Z$  then it is clear from the definitions that we have
$$T_{w, K} M = T_{w, \Z} M_\Z \otimes _\Z K.$$ 
Note that if $\lambda \in X$ then $\Delta_p(\lambda)$ has a $\Z$-form, namely $\Delta_\Z(\lambda)$.  Likewise $\nabla_\Z(\lambda)$ is a $\Z$-form of $\nabla_p(\lambda)$.

\section{Filtrations and sum formulas for antidominant Verma modules in $\mathcal O_p$}
As an application of the twisting functors introduced in Section 5 we shall in this section use them to construct Jantzen-type filtrations of all antidominant Verma modules in $\mathcal O_p$ and prove that they satisfy sum formulas analogous to the Weyl modules for $G$. Our approach is related to the one used in \cite{AL}  for category $\mathcal O$, but since we are in characteristic $p$ there are closer similarities to the modular case in \cite{A83} (or \cite{RAG}, Chapter II.8).

\subsection{The $\mathfrak{sl}_2$ case} \label{sl2}
We begin with the case $\mathfrak g = \mathfrak {sl}_2$. Here we choose the usual basis $\{f, h, e\}$ for this Lie algebra. This is also a Chevalley basis of $\mathfrak g_\Z$ and we shall work with the corresponding divided power basis of $U_\Z$. The Weyl group has order $2$ and we denote its generator by $s$.

Let $\lambda \in X = \Z$ and consider the twisted Verma module $T_{s, \Z} \Delta_\Z(\lambda)$ which we by Proposition \ref{twistZ} can identify with $\nabla_\Z(-\lambda -2)$. Write $\mu = - \lambda -2$ and observe that since $\mu$ is the maximal weight of $\nabla_\Z(\mu)$  we have $\Hom_{\mathcal O_\Z} (\Delta_\Z(\mu), T_{s, \Z} \Delta_\Z(\lambda)) \simeq \Z c_\mu$ for some (unique up to sign) homomorphism $c_\mu : \Delta_\Z(\mu) \to T_{s, \Z} \Delta_\Z(\lambda) \simeq \nabla_\Z(\mu)$.  

We now choose a generator $v_0$ of $\Delta_\Z(\mu)_\mu$. As this weight space equals $\Z$ we have that $v_0$ is unique up to sign. Then for $i \geq 0$ we set $v_i = f^{(i)}v_0$, so that $\{v_i| i \geq 0\}$ is a basis of $\Delta_\Z(\mu)_\mu$ and we have a corresponding dual basis $\{v^*_i \mid  i \geq 0\}$ of $\nabla_\Z(\mu)$. An elementary computation shows then that $c_\mu$ is given by (up to sign)
$$ c_\mu: v_i \mapsto \binom{\mu}{i} v_i^*, \; i \geq 0.$$

Let $C(\mu)$ denote the cokernel of $c_\mu$. The above formula for $c_\mu$ shows that $C(\mu) = \bigoplus_{i\geq 0} C(\mu)_{\mu - 2i}$ with 
\begin{equation} \label{sum sl_2} C(\mu)_{\mu - 2i} \simeq \Z/(\binom{\mu}{i})\Z.
\end{equation}
as $\Z$-modules.
\begin{lem} Suppose $\mu < 0$ and set $\lambda = -\mu -2$. Then  
$$ \cha_p^t C(\mu) = \sum_{i> 0} \nu_p(\binom{\mu}{i}) e^{\mu - 2i} =  \sum_{a > \lambda +1} \nu_p(a) \cha \Delta_p(\lambda - 2a) - \sum_{b > 0} \nu_p(b) \cha \Delta_p(\mu - 2b).$$
\end{lem}
\begin{proof} The first equality is immediate from (\ref{sum sl_2}) and the second drops out when we write the Verma characters in terms of the $e^{\mu -2i}$'s.
\end{proof}

\subsection{The minimal parabolic case}

We return to general $\mathfrak g$ and fix a simple root  $\alpha$.  Associated with $\alpha$ we have the minimal parabolic subalgebra $\mathfrak p_\alpha = \mathfrak g_{- \alpha} \oplus \mathfrak b$.

Consider $\lambda \in X$ and set $\mu = s \cdot \lambda$ where $s = s_\alpha$ denotes the simple reflection in $W$ associated with $\alpha$. 
Then in analogy with the $\mathfrak sl_2$ case in the previous section we take a generator $c_\mu^\alpha$ for the hom-space from $\Delta_{\alpha, \Z} (\mu)$ to $T_{s,\Z} \Delta_{\alpha, \Z}(\lambda)$, 
where for any $\nu \in X$ we have denoted by $\Delta_{\alpha, \Z} (\nu)$ the Verma module for $U_\Z(\mathfrak p_\alpha)$, i.e.  $\Delta_{\alpha, \Z} (\nu) = U_\Z(\mathfrak p_\alpha) \otimes_{B_\Z} \nu$.  

If $\langle \mu, \alpha^\vee \rangle < 0$ then we have with $\lambda = s \cdot \mu$ (using the same computations as in Section \ref{sl2}) a short exact sequence
\begin{equation} \label{alphases}  0 \to \Delta_{\alpha, \Z}(\mu) \to T_{s, \Z} \Delta_{\alpha, \Z}(\lambda) \to C_\alpha(\mu) \to 0
\end{equation}
where the first map is $c_\mu^\alpha$ and $C_\alpha(\mu)$ is a torsion module with 
$$ \cha_p^t C_\alpha(\mu) =  \sum_{a > - \langle  \mu +\rho, \alpha^\vee \rangle} \nu_p(a) \cha \Delta_{\alpha, p}(\lambda - a\alpha) - \sum_{b > 0} \nu_p(b) \cha \Delta_{\alpha,p}(\mu - b\alpha).$$
Here $\Delta_{\alpha, p}(\nu)$ is the Verma module for $U_K(\mathfrak p_\alpha)$ with highest weight $\nu$.
 
 When we apply the exact functor $U_\Z \otimes_{U_\Z(\mathfrak p_\alpha)} - $ to (\ref {alphases}) we get the short exact sequence
 \begin{equation} \label{globalses} 0 \to  \Delta_\Z(\mu) \to T_{s, \Z} \Delta_\Z(\lambda) \to C_s(\mu) \to 0
 \end{equation}
 where $C_s(\mu) = U_\Z \otimes_{U_\Z(\mathfrak p_\alpha)} C_\alpha(\mu)$ is a torsion module with 
 \begin{equation} \label{globalch} \cha_p^t C_s(\mu) =  \sum_{a > -\langle \mu + \rho, \alpha^\vee \rangle} \nu_p(a) \cha \Delta_p(\lambda - a\alpha) - \sum_{b > 0} \nu_p(b) \cha \Delta_{p}(\mu - b\alpha).
 \end{equation}
 Note that we have seriously abused notation: in (\ref{alphases}) the symbol $T_{s, \Z}$ denotes the twisting functor on $U_\Z(\mathfrak p_\alpha)$-modules  while in (\ref{globalses}) the same symbol denotes the twisting functor on $U_\Z$-modules.
 
 \subsection{The general case} \label{gencase}
 Consider now an element $w \in W$ for which $ws > w$. Then by Proposition \ref{twistZ}(2) we have $T_{ws, \Z} = T_{w, \Z} \circ T_{s, \Z}$ so that when we apply the right exact functor  $T_{w, \Z}$ to (\ref{globalses}) we obtain a  short exact sequence
 $$ 0 \to  T_{w, \Z}\Delta_\Z(\mu) \to T_{ws, \Z} \Delta_\Z(\lambda) \to T_{w,\Z}C_s(\mu) \to 0.$$
In fact, when tensored by $\C$ the two first terms are both isomorphic to the simple Verma module $\Delta(\mu) \in \mathcal O$. This gives the injectivity of the first map.  The last term in this sequence is a torsion module and when we  combine  Proposition \ref{twistZ}(3) and (\ref{globalch}) we get 
 $$ 
 \cha_p^t T_{w, \Z} C_s(\mu) =  \sum_{a > - \langle \mu + \rho, \beta^\vee \rangle} \nu_p(a) \cha \Delta_p(ws\cdot \mu - a\beta) - \sum_{b > 0} \nu_p(b) \cha \Delta_{p}(w \cdot \mu - b \beta),
 $$
 where $\beta = w(\alpha) \in R^+$.
 
 \subsection{Filtrations of antidominant Verma modules in $\mathcal O_p$}
 The results in the preceding subsections lead to the following result.
 \begin{thm} \label{sum}Let $\mu \in X^-$. The Verma module $\Delta_p(\mu) \in \mathcal O_p$ has a filtration
 $$ \cdots \subset \Delta_p^2(\mu) \subset \Delta_p^1(\mu) \subset \Delta_p(\mu)$$
 in which the submodules $\Delta_p^j(\mu)$ have the following properties
 \begin{enumerate}
 \item $\Delta_p(\mu)/\Delta_p^1(\mu) \simeq L_p(\mu)$.
 \item  The family $(\cha \Delta_p^j(\mu))_{j \geq 1}$ is summable and we have
 $$ \sum_{j \geq 1} \cha \Delta_p^j(\mu) = \sum_{\beta \in R^+} (\sum_{a_\beta > - \langle \mu + \rho, \beta^\vee \rangle} \nu_p(a_\beta) \cha \Delta_p(s_\beta \cdot \mu - a_\beta \beta) - \sum_{b_\beta > 0} \nu_p(b_\beta) \cha \Delta_p(\mu - b_\beta \beta)).$$
 \end{enumerate}
 \end{thm}
 
 \begin{proof}
 We have $\Hom_{\mathcal O_\Z}(\Delta_\Z(\mu), \nabla_\Z (\mu)) \simeq \Z$. Choose a generator $c(\mu)$ of this $\Z$-module. Note that $c(\mu)$ is unique up to sign and an isomorphism on the $\mu$ weight space. We then define for $j \geq 0$
 $$ \Delta_\Z^j(\mu) = c(\mu)^{-1}(p^j \nabla_\Z(\mu))$$
and 
$$\Delta_p^j(\mu) =\text { the  $K$-span of the image in } \Delta_p(\mu) = \Delta_\Z(\mu) \otimes_\Z K \text { of } \Delta_\Z^j(\mu).$$
If  $\overline{c(\mu)} = c(\mu) \otimes_\Z K$, then the image of $\overline {c(\mu)}$ equals $L_p(\mu)$ and we have $\Delta_p^1(\mu) = \ker(\overline{c(\mu)})$. Hence (1) is immediate.

To check (2) we shall now see that we may factor $c(\mu)$ into a composition of the homomorphisms occuring in the previous section. We let $w_0 = s_{i_1} s_{i_2} \cdots s_{i_N}$ be a reduced expression for the longest element in $W$ and set $y_j = s_{i_1} \cdots s_{i_j}, \; j=0, 1, ,\cdots, N$. Then $y_0 = 1$ and $y_N = w_0$. Moreover, $\beta_j = y_{j-1}(\alpha_{i_j}) \in R^+$ and in fact $R^+ = \{\beta_1, \cdots, \beta_N\}$.

By Section \ref{gencase} we have for each $j= 1,  \cdots, N$ a homomorphism 
$$ c_j(\mu): T_{y_{j-1}} \Delta_\Z(y_{j-1}^{-1} \cdot \mu) \to  T_{y_{j}} \Delta_\Z(y_{j}^{-1} \cdot \mu),$$
which is an isomorphism on the highest weight space (a rank $1$ space with weight $\mu$). Moreover, if $C_j(\mu)$ denotes the cokernel of $c_j(\mu)$ then $C_j(\mu)$ is a torsion $\Z$-module and we have
\begin{equation} \label{jcha} \cha_p^t C_j(\mu) = \sum_{a_j > -\langle \mu + \rho, \beta_j^\vee \rangle} \nu_p(a_j) \cha \Delta_p(s_{\beta_j} \cdot \mu - a_j \beta_j) - \sum_{b_j > 0} \nu_p(b_j) \cha \Delta_p(\cdot \mu - b_j \beta_j).
\end{equation}

The composite $c_N(\mu) \circ c_{N-1}(\mu) \circ \cdots c_1(\mu): \Delta_\Z(\mu) \to T_{w_0, \Z} \Delta_\Z(w_0\cdot \mu) \simeq \nabla_\Z(\mu)$ is an isomorphism on the $\mu$ weight space and hence coincides up to sign with $c(\mu)$. Via the short exact  sequences $0 \to T_{y_{j-1}} \Delta_\Z(y_{j-1}^{-1} \cdot \mu) \to  T_{y_{j}} \Delta_\Z(y_{j}^{-1} \cdot \mu) \to C_j(\mu) \to 0$, cf. Section (\ref{gencase}), we then get
$$\cha_p^t C(\mu) = \sum_{j=1}^N \cha_p^t C_j(\mu).$$
Inserting the formulas from (\ref{jcha}) we then see that $\cha_p^t C(\mu)$ coincides with the right hand side of (2). Now standard arguments as in \cite{A83} (or \cite{RAG}, II.8.18) give that $\cha_p^t C(\mu)$ is also equal to the left hand side of (2).
 \end{proof}
 
 \begin{rem} Note that the assumption that $\mu$ is antidominant in Theorem \ref{sum} is needed in the proof to ensure that the homomorphisms $c_j(\mu)$ are injective (and have cokernels which are torsion $\Z$-modules). In fact, the composite $c(\mu) : \Delta_\Z(\mu) \to \nabla_\Z(\mu)$ is injective iff $\mu$ is antidominant.
 \end{rem}
 
 \begin{exa}
 \begin{enumerate}
 \item If in Theorem \ref{sum} we take $\mu = - \rho$ then we see that the right hand side in the sum formula is zero (in fact, for each $\beta \in R^+$ the $\beta$ summand is $0$ as $s_\beta \cdot (-\rho) = -\rho$). This is therefore another proof of the fact, that $\Delta_p(-\rho)$ is irreducible, see \cite{An22}, Theorem 6.1. More generally,  in the sum formula corresponding to an arbitrary  $\mu \in X^-$ the only $\beta \in R^+$ that contribute non-trivially to the sum are those with $\langle \mu + \rho, \beta^\vee \rangle < 0$.
 
 \item Consider $\mathfrak g = \mathfrak{sl}_2$ and take $\mu = -\alpha$. We claim that the filtration in Theorem \ref{sum} in this case is $\Delta_p^j(-2) = \Delta_p(-2)^{(j)}$,  $j \geq 0$. This follows from the considerations in Section 5.2 in \cite{An22}. In fact, the sequence (5.8) there shows that $\Delta_p(-2)^{(1)} = \Delta_p^1(-2)$. To check that the same is true for all $j$ we just observe that the character sum formula for the submodules $(\Delta_p(-2)^{(j)})_{j \geq 1}$ coincides with the one for $(\Delta_p^{(j)}(-2))_{j \geq 1}$ given in Theorem \ref{sum} (2). Now all non-zero weight multiplicities in $\Delta_p(-2)$ are $1$. Hence any two filtrations by submodules with the same character sum must coincide.
 
 \end{enumerate}
 \end{exa}

 \section{Character formulas for indecomposable tilting modules}
 Recall from \cite{An22} that in $\mathcal O_p$ we have two kinds of tilting modules: (1) the finite dimensional ones  and (2) the $\infty$-dimensional ones defined as follows
 \begin{enumerate}
 \item A module $Q \in \mathcal C$ is called a tilting module if it has a filtration by Weyl modules as well as by dual Weyl modules,
 
 \item A module $Q \in \mathcal O_p$ is called $\infty$-tilting if it has a finite filtration by Verma modules as well as by dual Verma modules.
 \end{enumerate}

Both kinds of tilting modules split into indecomposable modules.  In this section we have collected some character formulas giving relations among indecomposable tilting modules. Our tilting character formulas are immediate consequences of results from \cite{An22} via our setup from the previous sections. Therefore we have just stated the results leaving the immediate proofs to the readers. In contrast to the situation in Section 2 our formulas does not reduce the problem of finding all indecomposable tilting characters to a finite set. We shall end the section (and the paper) by some remarks on how one  nevertheless have some useful reductions.

\subsection{Characters for indecomposable finite dimensional tilting modules}

Up to isomorphisms we have one indecomposable  tilting module $T_p(\lambda) \in \mathcal C$ for each $\lambda \in X^+$. The weights of $T_p(\lambda)$ are all less than or equal to $\lambda$, and $\lambda$ occurs with multiplicity $1$. These modules have been well studied and results can be found in the literature. For future use we just mention the following results
\begin{enumerate}
\item The linkage principle in $\mathcal C$ shows that if $L_p(\mu)$ is a composition factor (in the usual sense) of $T_p(\lambda)$ for some $\mu, \lambda \in X^+$ then $\mu \in W_p \cdot \lambda$.

\item Let $p \geq h$. Translation functor arguments show that in order to determine $\cha T_p(\lambda)$ for all $\lambda \in X^+$ it is enough to do so for $\lambda \in W_p \cdot 0 \cap X^+$. 

\item  Let $r \geq 1$ and $p \geq 2h-2$ . Then for  $\lambda \in (p^r - 1)\rho + X_r$  and arbitrary $\mu \in X^+$ we have $\cha T_p(\lambda + p^r \mu) = \cha T_p(\lambda) \star \cha^{(r)} T_p(\mu)$.

\end{enumerate}
The results in (1), (2) and (3) are ``classical'' and can all be found in \cite{RAG}, II.E.

\begin{rem} The result (3) above is called Donkin's tensor product theorem and was proved by him in \cite{Do}, Proposition 2.1. It may hold for smaller values of $p$ than the bound we have stated, but recently several examples by Bendel, Nakano, Pillen and Sobaje for $p=2$ and $p=3$, \cite{BNPS1} and \cite{BNPS2}, show that there are exceptions. In a very recent preprint \cite{BNPS3} these authors show that the bound on $p$ can in general be lowered to $2h-4$. 

We emphasize that although (3) gives the characters of indecomposable tilting modules for certain ``large'' $\nu \in X^+$ in terms of smaller weights it does not reduce the problem of determining all  tilting characters to a finite set of weights (as we saw was the case for irreducible characters in Section 2). There are large infinite subsets of $X^+$ consisting of weights that cannot be expressed as in (3).
\end{rem}

\begin{rem} Note that by the recent results in \cite{AMRW} and \cite{RW1} the characters of all finite dimensional tilting modules are determined by the $p$-Kazhdan-Lusztig polynomials!
\end{rem}

\subsection{Characters of indecomposable $\infty$-tilting modules}

As demonstrated in \cite{An22}, Theorem 7.5 we have in analogy with the finite dimensional case that up to isomorphisms we have one indecomposable tilting module $T^\infty_p(\lambda) \in \mathcal O_p$ for each $\lambda \in X^+ -\rho$. The weights of $T_p^\infty(\lambda)$ are all less than or equal to $\lambda$, and $\lambda$ occurs with multiplicity $1$. We also have
\begin{enumerate}
\item The linkage principle in $\mathcal O_p $ (combine Corollary \ref{SLP} with Theorem 8.8 in \cite{An22})  shows that if $L_p(\mu)$ is a composition factor of $T_p^\infty(\lambda)$ for some $\mu \in X, \lambda \in X^+ - \rho$ then $\mu \in W_p \cdot \lambda$

\item Let $p \geq h$. Translation functor arguments (from Section \ref{trans}) show that in order to determine $\cha T_p^\infty(\lambda)$ for all $\lambda \in X^+ - \rho$ it is enough to do so for $\lambda \in W_p \cdot 0 \cap (X^+ - \rho)$. 

\item  Let  $p \geq 2h-2$. If $r \geq 1$ and $\lambda \in X_r - \rho$  then $\cha T_p^\infty (\lambda) = \cha T_p(\lambda + p^r \rho) \star \cha^{(r)} \Delta_p(-\rho)$. This is an easy consequence of \cite{An22}, Theorem 7.7(2).

\end{enumerate}

\begin{rem} As the character of $\Delta_p(-\rho)$ is known (3) tells us that all $\infty$-tilting characters are determined by the finite dimensional tilting characters when $p \geq 2h-2$. This bound on $p$ can sometimes but not always be relaxed, see \cite{An22}, Remark 7.8. For all $p$ we have the inequality $\cha T_p^\infty (\lambda) \leq \cha T_p(\lambda + p^r \rho) \star \cha^{(r)} \Delta_p(-\rho)$ because of \cite{An22}, Theorem 7.7(1). Here the inequality $f \leq g$ between two maps $f, g: X \to \Z$ means $f(\nu) \leq g(\nu)$ for all $\nu \in X$.
\end{rem}

\vskip 1 cm
\end{document}